\documentclass[oneside,12pt]{amsart}
\usepackage{enumerate,amsfonts,amsmath,amsthm,amssymb,latexsym,verbatim,amscd,mathrsfs}
\usepackage[usenames]{color}
\usepackage{hyperref}

\usepackage[left=3cm,right=2cm,top=2cm,bottom=2cm]{geometry}

\usepackage{fancyhdr}
\numberwithin{equation}{section}

\theoremstyle{definition}
\newtheorem{thm}{Theorem}[section]
\newtheorem{prop}[thm]{Proposition}

\newtheorem{lem}[thm]{Lemma}

\newtheorem{rem}[thm]{Remark}
\newtheorem{conj}[thm]{Conjecture}

\newcommand{\wt}{\mathrm{wt}}

\def\cal#1{\text{$\mathcal{#1}$}}
\def\lie#1{\text{$\mathfrak{#1}$}}

\newcommand{\e}{\quad \text{ and } \quad}

\newcommand{\ev}{\mathrm{ev}}

\newcommand{\ch}{\text{ch}}

\begin{document}

\title[Representations of hyper multicurrent and multiloop algebras]{Finite-dimensional representations of\\ hyper multicurrent and multiloop algebras}  
\author{Angelo Bianchi} 
\address{Universidade Federal de S\~ao Paulo - UNIFESP - Department of Science and Technology, Brazil}
\email{acbianchi@unifesp.br}
\thanks{A.B. is partially supported by CNPq grant 462315/2014-2 and FAPESP grants 2015/22040-0 and 2014/09310-5.}

\author{Samuel Chamberlin}
\address{Department of Mathematics and Statistics\\
Park University\\
Parkville, MO 64152}
\email{samuel.chamberlin@park.edu}


\ 

\begin{abstract}
We investigate the categories of finite-dimensional representations of multicurrent and multiloop hyperalgebras in positive characteristic, i.e., the hyperalgebras associated to the multicurrent algebras $\lie g\otimes\mathbb{C}[t_1,\ldots,t_n]$ and to the multiloop algebras $\lie g\otimes\mathbb{C}[t_1^{\pm1},\ldots,t_n^{\pm 1}]$, where $\lie g$ is any finite-dimensional complex simple Lie algebra. The main results are the construction of the universal finite-dimensional highest-weight modules and a classification of the irreducible modules in each category. In the characteristic zero setting we also provide a relationship between these modules.
\end{abstract}

\maketitle

\section*{Introduction}

The category of level zero representations of affine and quantum affine algebras and its full subcategories of finite-dimensional representations have been extensively studied recently. In the representation theory of such algebras an important family of representations of the loop and current algebras are the universal finite-dimensional highest weight modules nowadays called \textit{local Weyl modules}. These modules were introduced by Chari and Pressley in \cite{CPweyl} and became objects of independent interest (cf. \cite{C,cfk,CG,CL,CPweyl,FL,fkks,fl07,JM2,nss2} and references therein). We highlight two of these papers: in \cite{FL} the authors considered (local and global) Weyl modules for certain \textit{map Lie algebras} extending partial results from \cite{CPweyl}; and in \cite{cfk} the authors provided a categorical approach to (local and global) Weyl modules in the context of more general map algebras.

During the last decade, there has been intense study of the finite-dimensional representation theory of such map algebras, which have the form $\mathfrak{g}\otimes_{\mathbb C} \mathcal{A}$, where $\mathfrak{g}$ is a finite-dimensional complex simple Lie algebra and $\mathcal{A}$ is a commutative, associative, unital $\mathbb C$-algebra. These algebras can be regarded as generalizations of the (centerless untwisted) affine Kac-Moody algebras, which are recovered in the particular case that $\mathcal{A}=\mathbb C[t,t^{-1}]$. 
In addition, the multicurrent and multiloop algebras (where $\mathcal A=\mathbb{C}[t_1,\ldots,t_n]$ or $\mathcal A=\mathbb{C}[t_1^{\pm1},\ldots,t_n^{\pm 1}]$ respectively) are mutually related, since the multicurrent algebras can be seen as subalgebras of the multiloop algebras, which are essential to the construction of the toroidal Lie algebras (see \cite{PB,K2,L,R} and references therein).

In order to extend some results of the aforementioned papers to arbitrary algebraically closed fields $\mathbb F$ we need \textit{hyperalgebras}, which are constructed by considering an \textit{integral form} of the universal enveloping algebra $U(\lie g\otimes\mathcal{A})$, and then tensoring this form over $\mathbb Z$ with $\mathbb F$. Given a Lie algebra $\lie a$ over $\mathbb C$, the corresponding hyperalgebra is denoted by $U_\mathbb F(\lie a)$. In the case of the finite-dimensional complex simple Lie algebra $\lie g$, one considers Konstant's integral form for $U(\lie g)$ \cite{kosagz}. The affine analogues of Kostant's form were obtained by Garland \cite{G} for the non-twisted affine Kac-Moody algebras and by Mitzman \cite{mitz} in a more general way for all affine Kac-Moody algebras. Integral forms for (the universal enveloping algebra of) general map algebras $U(\lie g\otimes\mathcal{A})$ were formulated by the second author in \cite{sam1}.

The motivation to consider hyperalgebras comes from the  following fact: if $\lie g$ is a semisimple finite-dimensional Lie algebra over the complex numbers, $\mathbb F$ is an algebraically closed field, and $G_\mathbb F$ is a connected, simply connected, semisimple algebraic group over $\mathbb F$ of the same Lie type as $\lie g$, then the category of finite-dimensional $G_\mathbb F$-modules is equivalent to that of the hyperalgebra $U_\mathbb F(\lie g)$. Further, when the characteristic of $\mathbb F$ is zero, the hyperalgebra coincides with the universal enveloping algebra, as explained in Subsection \ref{s:intform}.

The finite-dimensional representations of the hyperalgebras in the (untwisted) affine case were studied by Jakeli\'c and Moura in \cite{JM1}. In \cite{bm}, the first author and Moura dealt with the twisted affine case. Hyperalgebras in the current algebra case and their connections with the loop algebra case and Demazure theory were studied by the first author, Macedo and Moura in \cite{bmm}.

The goal of this paper is to establish basic results about the finite-dimensional representations of multicurrent and multiloop hyperalgebras extending some of the known results in the case of current and loop algebras ($\lie g\otimes \mathbb{C}[t]$ and $\lie g\otimes \mathbb{C}[t^{\pm1}]$ resposectively). The approach for this is similar to \cite{JM1,bmm}, with the caveat that the characteristic zero methods as in \cite{cfk} are not available for the hyperalgebra setting.  We denote the multicurrent algebra in $n$ variables by $\lie g[n]$ and the multiloop algebra by $\lie g\langle n\rangle$. Our main results are the construction of the universal finite-dimensional highest-weight modules in the categories of finite-dimensional $U_\mathbb{F}(\lie g[n])$-modules and $U_\mathbb{F}(\lie g\langle n\rangle)$-modules, which we call \textit{graded local Weyl modules} and \textit{local Weyl modules}, respectively. We also provide a classification of the irreducible modules in each of these categories. 

Section \ref{pre} is dedicated to the algebraic preliminaries, including a review of the relevant facts about integral forms and hyperalgebras.   
In Section \ref{modules} we focus on categories of certain finite-dimensional modules. In Subsection \ref{sec:rev.g} we review the finite dimensional representation theory of hyperalgebras of the form $U_\mathbb{F}(\lie g)$. In Subsection \ref{sec:rev.gloop} we introduce the category of finite dimensional $U_\mathbb F(\lie g\langle n\rangle)$-modules, introduce the $n$-Drinfeld polynomials, define the local Weyl modules, and, in Theorem \ref{mainthm2}, we prove that they are the universal highest weight objects in this category and classify the irreducible objects. In Subsection \ref{ssfdirreps} we present a classification of the irreducible $U_{\mathbb F}(\lie g[n])$-modules similarly to \cite{CG}, define the graded local Weyl modules and, in Theorem \ref{mainthm}, we prove that they are the universal highest weight objects in this category. In Subsection \ref{iso} we discuss a conjecture (Conjecture \ref{conjec}) on a possible isomorphism between local and graded local Weyl modules and establish that in the characteristic zero setting there is a relationship between them as a quotient (Proposition \ref{quotient}). Some of the main proofs are presented separately in Section \ref{mainproofs}.

\ 

\noindent
\textbf{Acknowledgment.}  The authors are thankful to Dr. Adriano Moura for discussions about Drinfeld polynomials. 
\tableofcontents

\section{Preliminaries}\label{pre}

Throughout this work we denote by $\mathbb C$, $\mathbb Z$, $\mathbb Z_+$ and $\mathbb N$ the sets of complex numbers, integers, nonnegative integers and positive integers, respectively.

\subsection{Simple Lie algebras}\label{s:simple}

Let $\lie g$ be a finite-dimensional complex simple Lie algebra. Fix a Cartan subalgebra $\lie h$ of $\lie g$ and let $R$ denote the corresponding set of roots. Let $I$ be the set of vertices of the Dynkin diagram associated to $\lie g$. Let $\{\alpha_i:{i\in I}\}$ (respectively $\{ \omega_i :{i\in I}\}$) denote the simple roots (respectively fundamental weights) and set $Q=\bigoplus_{i\in I} \mathbb Z \alpha_i$, $Q^+=\bigoplus_{i\in I} \mathbb Z_+ \alpha_i$, $P=\bigoplus_{i\in I} \mathbb Z \omega_i$, $P^+=\bigoplus_{i\in I} \mathbb Z_+ \omega_i$, and $R^+=R\cap Q^+$.

Let $\mathcal{C}:=\{x_\alpha^{\pm},h_i:\alpha\in R^+,\ i\in I\}$ be a Chevalley basis of $\lie g$ and set
$x_i^\pm=x_{\alpha_i}^{\pm}$, $h_\alpha=[x_\alpha^+,x_\alpha^-]$, and $h_i=h_{\alpha_i}$. For each $\alpha\in R^+$, the subalgebra of $\lie g$ spanned by $\{x_{\alpha}^\pm,h_\alpha\}$ is isomorphic to $\lie{sl}_2$. 

Set
$$\lie n^\pm=\bigoplus_{\alpha\in R^+}\mathbb{C} x_\alpha^\pm,$$
and note that $\lie g=\lie n^-\oplus \lie h\oplus \lie n^+.$ 
We denote the Weyl group by $\cal W$, its longest element by $w_0$, and the highest root by $\theta$.

\subsection{Multicurrent and multiloop algebras} 

Let $\mathbb{C}[n]:=\mathbb{C}[t_1,\ldots,t_n]$ be the polynomial ring in the variables $t_1,\dots,t_n$ and let $\mathbb{C}\langle n\rangle:=\mathbb{C}\left[t_1^{\pm1},\ldots,t_n^{\pm1}\right]$ be the Laurent polynomial ring also in the variables $t_1,\dots,t_n$. We define the \textit{multicurrent algebra} associated to $\lie g$ as $\lie g[n]:=\lie g\otimes\mathbb C[n]$ and the \textit{multiloop algebra} as 
$\lie g\langle n\rangle:=\lie g\otimes\mathbb C\langle n\rangle$, where the bracket is given by $[x\otimes f, y\otimes g] =[x,y]\otimes fg$ for $x,y\in\lie g$ and $f,g\in\mathbb C\langle n\rangle$. Notice that $\lie g\otimes 1$ is a subalgebra of both $\lie g[n]$ and $\lie g\langle n\rangle$, which is isomorphic to $\lie g$. If $\lie a$ is a subalgebra of $\lie g$, then $\lie a[n]:=\lie a\otimes \mathbb C[n]$ and $\lie a\langle n\rangle:=\lie a\otimes \mathbb C\langle n\rangle$ are naturally subalgebras of $\lie g[n]$ and $\lie g\langle n\rangle$ respectively. In particular, we have 
$$\lie g[n] = \lie n^-[n] \oplus \lie h[n] \oplus \lie n^+[n]
\e \lie g\langle n\rangle=\lie n^-\langle n\rangle\oplus\lie h\langle n\rangle\oplus\lie n^+\langle n\rangle$$
and $\lie h[n]$ and $\lie h\langle n\rangle$ are abelian subalgebras of $\lie g[n]$ and $\lie g\langle n\rangle$ respectively.

We set 
$$\mathbb{C}[n]_0:=\{f\in \mathbb{C}[n] \mid f(0,\dots,0)=0\} \qquad \text{ and } \qquad \mathbb{C}[n]_+:=\mathbb{C}[n]\setminus \mathbb C.$$
Then we define
$$\lie{h}[n]_0:=\lie h\otimes \mathbb{C}[n]_0\subset \lie{h}[n]\qquad \text{ and } \qquad \lie{g}[n]_+:=\lie g\otimes \mathbb{C}[n]_+\subset \lie{g}[n]$$

\subsection{Universal enveloping algebras} 

For a Lie algebra $\lie a$, we denote by $U(\lie a)$ the corresponding universal enveloping algebra of $\lie a$. The Poincar\'e--Birkhoff--Witt (PBW) Theorem implies the isomorphisms
\begin{align*}
U(\lie g)&\cong U(\lie n^-)\otimes U(\lie h)\otimes U(\lie n^+)\\
U(\lie g[n])&\cong U(\lie n^-[n])\otimes U(\lie h[n])\otimes U(\lie n^+[n])\\
U(\lie g\langle n\rangle)&\cong U(\lie n^-\langle n\rangle)\otimes U(\lie h\langle n\rangle)\otimes U(\lie n^+\langle n\rangle).
\end{align*}

The assignments $\Delta:\lie g\langle n\rangle\to U(\lie g\langle n\rangle)\otimes U(\lie g\langle n\rangle)$ where $x\mapsto x\otimes 1+1\otimes x$, $S:\lie g\langle n\rangle\to \lie g\langle n\rangle$ where $x\mapsto-x$, and $\epsilon:\lie g\langle n\rangle\to\mathbb C$ where $x\mapsto0$, can be  uniquely extended so that $U(\lie g\langle n\rangle)$ becomes a Hopf algebra with comultiplication $\Delta$, antipode $S$, and counit $\epsilon$.

Notice that the restriction of these maps to $\lie g[n]$ also induces a Hopf algebra structure to $U(\lie g[n])$. For any Hopf algebra $H$, we denote by $H^0$ its augmentation ideal (the kernel of $\epsilon$).

\subsection{Integral forms and hyper algebras} \label{s:intform}

Given $\alpha\in R$ and $f\in \mathbb{C}\langle n\rangle\setminus\{0\}$, consider the following power series with coefficients in $U(\lie h\langle n\rangle)$:
\begin{equation*}  
	\Lambda_{\alpha,f}(u) = \sum_{r=0}^\infty \Lambda_{\alpha,f,r}u^r = \exp \left(-\sum_{s=1}^{\infty}\frac{h_{\alpha}\otimes f^{s}}{s} u^{s}\right).
\end{equation*}
Note that $\Lambda_{\alpha,f,r}$ is a polynomial in $\left(h_\alpha\otimes f^j\right)$ for $j\in\{1,\ldots,r\}$. For $i\in I$, we shall write $\Lambda_{i,f,r}$ in place of $\Lambda_{\alpha_i,f,r}$. 

\begin{rem} This series is a natural extension from the series introduced in \cite{G}.
\end{rem}

Let $$\mathbb{B}[n]:=\left\{t_1^{r_1}\cdots t_n^{r_n}\ \big|\ (r_1,\ldots,r_n)\in\mathbb{Z}_+^n\right\} \qquad \text{and} \qquad \mathbb{B}\langle n\rangle:=\left\{t_1^{r_1}\cdots t_n^{r_n}\ \big|\ (r_1,\ldots,r_n)\in\mathbb{Z}^n\right\}$$
be the canonical basis for $\mathbb{C}[n]$ and $\mathbb{C}\langle n\rangle$ respectively. 

 Given $\mathbf s=\left(s_1,\ldots,s_n\right)\in\mathbb Z^n$ and $\alpha\in R^+$ we shall denote  
 $\mathbf t^{\mathbf s}:=t_1^{s_1}\cdots t_n^{s_n}\in\mathbb C\langle n\rangle$ and $\Lambda_{i,\mathbf s,r}:=\Lambda_{i,\mathbf t^\mathbf s,r}, \ x_{\alpha,\mathbf s}^-:=x_{\alpha}^-\otimes\mathbf t^{\mathbf s}\in\lie g\langle n\rangle$. Similarly for elements in $\lie{g}[n]$.

Define $$\mathbb{B}^{\bullet}:=\left\{t_1^{r_1}\cdots t_n^{r_n}\ |\ (r_1,\ldots,r_n)\in\mathbb{Z}^n,\ \gcd(r_1,\ldots,r_n)=1\right\}$$ $$\text{and} \qquad  \mathbb{B}^{\bullet}_+:=\left\{t_1^{r_1}\cdots t_n^{r_n}\ |\ (r_1,\ldots,r_n)\in\mathbb{Z}_+^n,\ \gcd(r_1,\ldots,r_n)=1\right\}.$$ 
Note that $\mathbb{B}^\bullet$ is the set of basis elements for $\mathbb{C}\langle n\rangle$, which are not powers of other basis elements. Similarly for $\mathbb{B}^{\bullet}_+$ and $\mathbb{C}[n]$. 

The following lemma shows that elements of the form $\Lambda_{\alpha,f^k,r}$ are linear combinations of products of elements of the form $\Lambda_{\alpha,f,s}$.
\begin{lem}\label{f^kintermsoff}
Let $f\in\mathbb{B}^\bullet$, $\alpha\in R$, and $r,k\in\mathbb N$. Then
\begin{equation*}
	\Lambda_{\alpha,f^k,r}=k\Lambda_{\alpha,f,kr}+\sum_{(\mathbf s,\mathbf n)}m_{\mathbf s,\mathbf n}\Lambda_{\alpha,f,s_1}^{n_1}\dots\Lambda_{\alpha,f,s_l}^{n_l}
	\end{equation*}
    where $m_{\mathbf s,\mathbf n}\in\mathbb Z$ and the sum is over all $\mathbf s,\mathbf n\in\mathbb N^l$ for some $l\in\mathbb N$ such that $s_i\neq s_j$ for all $i\ne j,$ $l\sum n_j>1$, and $\sum n_js_j=rk$.  
\end{lem}
\begin{proof}
The lemma is proven in $\lie{g}\langle1\rangle$ for $f=t_1$ as part of Lemma 5.11 in \cite{G}. We can extend it to the current setting by replacing $t_1$ with $f$.
\end{proof}

The pure tensors in $\mathcal{C}\otimes\mathbb{B}[n]$ and $\mathcal{C}\otimes\mathbb{B}\langle n\rangle$ form bases for $\lie g[n]$ and $\lie g\langle n\rangle$ respectively, which we denote by $\mathcal{B}[n]$ and $\mathcal{B}\langle n\rangle$ respectively. 
Given an order on each of these sets $\mathcal{B}[n]$ and $\mathcal{B}\langle n\rangle$ and a PBW monomial with respect to this order, we construct an ordered monomial in the elements of the sets

$$\cal M[n]=\left\{(x^\pm_{\alpha}\otimes b)^{(k)}, \ \Lambda_{i,c,r}, \ \binom{h_{i}\otimes1}{k}\  \bigg|\ \alpha \in R^+,\ i\in I,\ b,c \in \mathbb B[n],\ c\ne 1,\  k,r\in\mathbb N\right\},$$ and
$$\cal M\langle n\rangle=\left\{(x^\pm_{\alpha}\otimes b)^{(k)}, \ \Lambda_{i,c,r}, \ \binom{h_{i}\otimes1}{k}\  \bigg|\ \alpha \in R^+,\ i\in I,\ b,c \in \mathbb B\langle n\rangle,\ c\ne 1 ,\ k,r\in\mathbb N\right\},$$
where
$$(x^\pm_{\alpha}\otimes b)^{(k)}:=\frac{(x^\pm_{\alpha}\otimes b)^{k}}{k!},\qquad  \binom{h_i\otimes1}{k}:=\frac{h_i\otimes 1(h_i\otimes1-1)\dots(h_i\otimes 1-k+1)}{k!},$$
via the correspondence 
$$(x^\pm_{\alpha}\otimes b)^k \leftrightarrow (x^\pm_{\alpha}\otimes b)^{(k)}, \quad (h_{i}\otimes1)^k \leftrightarrow \binom{h_{i}\otimes1}{k}, \quad \text{ and } \quad (h_{i}\otimes c)^{r} \leftrightarrow\Lambda_{i,c,r}.$$

Using the obvious similar correspondence we consider monomials in $U(\lie g)$ formed by elements of
$$\cal M=\left\{ (x_{\alpha}^\pm)^{(k)},\tbinom{h_i}{k}  :  \alpha\in R^+, i\in I, k\in\mathbb N\right\}.$$
It is clear that we have natural inclusions $\cal M\subset \cal M[n]\subset \cal M\langle n\rangle$. The set of ordered monomials thus obtained are bases of $U(\lie g)$, $U(\lie g[n])$, and $U(\lie g\langle n\rangle)$, respectively.

Let  $U_{\mathbb Z}(\lie g) \subseteq U(\lie g)$, $U_{\mathbb Z}(\lie g[n])\subseteq U(\lie g[n])$, and $U_{\mathbb Z}(\lie g\langle n\rangle)\subseteq U(\lie g\langle n\rangle)$ be the $\mathbb Z$--subalgebras generated respectively by
$$\{(x_{\alpha}^\pm)^{(k)}  :  \alpha\in R^+, k\in\mathbb Z_+\},\qquad \{(x^\pm_{\alpha}\otimes b)^{(k)} :  \alpha \in R^+, b\in \mathbb B[n], k\in\mathbb Z_+\}$$ 
$$\text{ and }\qquad \{(x^\pm_{\alpha}\otimes b)^{(k)} :  \alpha \in R^+, b\in \mathbb B\langle n\rangle, k\in\mathbb Z_+\}.$$

We have the following crucial theorem which was proved in \cite{kosagz}.

\begin{thm} \label{gforms}
The subalgebra $U_{\mathbb Z}(\lie g)$ is a free $\mathbb Z$-module and the set of ordered monomials constructed from $\cal M$ is a $\mathbb Z$-basis of $U_{\mathbb Z}(\lie g)$.
\end{thm}

The next theorem is a particular case of \cite[Theorem 3.2]{sam1} (see also \cite{G,mitz} for the affine Kac-Moody algebra case.)

\begin{thm} \label{gnforms}
The subalgebras $U_{\mathbb Z}(\lie g[n])$ and $U_{\mathbb Z}(\lie g\langle n\rangle)$ are free $\mathbb Z$-modules and the set of ordered monomials constructed from $\cal M[n]$ and $\cal M\langle n\rangle$ are $\mathbb Z$-bases of $U_{\mathbb Z}(\lie g[n])$ and $U_{\mathbb Z}(\lie g\langle n\rangle)$ respectively.
\end{thm}

In particular, if $\lie a\in\{\lie g,\lie n^\pm,\lie h,\lie g[n],\lie n^\pm[n],\lie h[n],\lie n^\pm\langle n\rangle,\lie h\langle n\rangle\}$ and we set
$$
U_\mathbb Z(\lie a) := U(\lie a)\cap U_\mathbb Z(\lie g\langle n\rangle)
$$
we get
$$
\mathbb C\otimes_\mathbb Z U_\mathbb Z(\lie a) \cong  U(\lie a).
$$
In particular, $U_\mathbb Z(\lie g\langle n\rangle)$, $U_\mathbb Z(\lie g[n])$, and $U_\mathbb Z(\lie g)$ are \textit{integral forms} of $U(\lie g\langle n\rangle)$, $U(\lie g[n])$, and $U(\lie g)$, respectively.

\begin{rem} \
\begin{enumerate} 
\item If $\lie a$ is as above, $U_\mathbb Z(\lie a)$ is a free $\mathbb Z$-module spanned by monomials formed by elements of $M\cap U(\lie a)$ for the appropriate $M\in\{\cal M,\cal M[n],\cal M\langle n\rangle\}$.
\item 
Notice that $U_\mathbb Z(\lie g)=U(\lie g)\cap U_\mathbb Z(\lie g\langle n\rangle)$, i.e., Kostant's integral form of $\lie g$ coincides with its intersection with the integral form of $U(\lie g\langle n\rangle)$ which allows us to regard $U_\mathbb Z(\lie g)$ as a $\mathbb Z$-subalgebra of $U_\mathbb Z(\lie g\langle n\rangle)$. We can regard $U_\mathbb Z(\lie g[n])$ as a $\mathbb Z$-subalgebra of $U_\mathbb Z(\lie g\langle n\rangle)$ similarly.
\end{enumerate}
\end{rem}

Given a field $\mathbb F$, the $\mathbb F$--hyperalgebra of $\lie a$ is defined by
$$U_\mathbb F(\lie a) :=  \mathbb F\otimes_{\mathbb Z}U_\mathbb Z(\lie a)$$
where $\lie a\in\{\lie g,\lie n^\pm,\lie h,\lie g[n],\lie n^\pm[n],\lie h[n],\lie g\langle n\rangle,\lie n^\pm\langle n\rangle,\lie h\langle n\rangle\}$. We will refer to $U_\mathbb F(\lie g\langle n\rangle)$ as the \textit{hyper multiloop algebra} of $\lie g$ in $n$-variables over $\mathbb F$ and $U_\mathbb F(\lie g[n])$ as the \textit{hyper multicurrent algebra }of $\lie g$ in $n$-variables over $\mathbb F$. 

The PBW Theorem implies
\begin{align*}\label{tridecomp}
U_\mathbb F(\lie g)&=U_\mathbb F(\lie n^-)U_\mathbb F(\lie h)U_\mathbb F(\lie n^+)\\
U_\mathbb F(\lie g[n])&=U_\mathbb F(\lie n^-[n])U_\mathbb F(\lie h[n])U_\mathbb F(\lie n^+[n])\\
U_\mathbb F(\lie g\langle n\rangle)&=U_\mathbb F(\lie n^-\langle n\rangle)U_\mathbb F(\lie h\langle n\rangle)U_\mathbb F(\lie n^+\langle n\rangle).
\end{align*}
We will keep denoting by $x$ the image of an element $x\in U_\mathbb Z(\lie a)$ in $U_\mathbb F(\lie a)$.

\begin{rem}\label{rem1} Let $\lie a\in\{\lie g[n],\lie g\langle n\rangle\}$
	\begin{enumerate}
		\item If the characteristic of $\mathbb F$ is zero, the algebra $U_\mathbb F(\lie a)$ is naturally isomorphic to $U(\lie a_\mathbb F)$, where $\lie a_\mathbb F=\mathbb F\otimes_\mathbb Z \lie a_\mathbb Z$, $\lie g[n]_\mathbb Z$ is the $\mathbb Z$-span of $\mathcal{B}[n]$, and $\lie g\langle n\rangle_\mathbb Z$ is the $\mathbb Z$-span of $\mathcal{B}\langle n\rangle$.
		\item If $\mathbb F$ has positive characteristic we have an algebra homomorphism $U(\lie a_\mathbb F)\to U_\mathbb F(\lie a)$ which is neither injective nor surjective.
		\item The Hopf algebra structure on $U(\lie a)$ induces a natural Hopf algebra structure over $\mathbb Z$ on $U_\mathbb Z(\lie a)$ and this in turn induces a Hopf algebra structure on $U_\mathbb F(\lie a)$.
	\end{enumerate}
	\end{rem}

\subsection{Straightening Identities}

The next lemmas are fundamental in the proofs of Theorems \ref{gforms} and \ref{gnforms} and will also be crucial in the study of finite-dimensional representations of multicurrent and multiloop hyperalgebras. Lemma \ref{basicrel} was originally proved in the $U(\lie g\otimes \mathbb C[t,t^{-1}])$ setting in \cite[Lemma 7.5]{G} and in \cite[Lemma 5.4]{sam1} for a general map algebra $U(\lie g\otimes \mathcal A)$. See also \cite[Lemma 1.3]{JM1} for the first presentation in the context of hyperalgebras.

\begin{lem} 
\label{basicrel}
Let $\alpha \in R^+$, $r,s \in \mathbb N$ such that $s\ge r\ge 1$, and $a,b\in\mathbb C\langle n\rangle\setminus\{0\}$. Then,
\begin{equation}
(x^+_{\alpha}\otimes a)^{(r)}(x^-_{\alpha}\otimes b)^{(s)}-(-1)^r \left(\left(X_{\alpha,a,b} (u)\right) ^{(s-r)}\Lambda_{\alpha,ab}(u)  \right)_s\in U_\mathbb F(\lie g\langle n\rangle)U_\mathbb F(\lie n^+\langle n\rangle)^0.
\end{equation}
In particular, if $ab\in\mathbb C[n]_+$,
then
\begin{equation}
(x^+_{\alpha}\otimes a)^{(r)}\left(x^-_{\alpha}\otimes b\right)^{(s)}-(-1)^r \left(\left(X_{\alpha,a,b} (u)\right) ^{(s-r)}  \right)_s\in U_\mathbb F(\lie g[n])\left(U_\mathbb F(\lie n^+[n])^0+U_\mathbb F(\lie h[n]_0)^0\right).
\end{equation}

\noindent
where $$X_{\alpha,a,b}(u)  = \sum_{m=0}^\infty(x^-_\alpha \otimes a^mb^{m+1})u^{m+1}$$ and the subscript $s$ above means the coefficient of $u^s$ in this power series.
\vspace{-.4cm}
\hfill\qedsymbol
\vspace{0.2cm}
\end{lem}

\

Given $\alpha \in R^+$, $a\in\mathbb{C}\langle n\rangle$ and $k\ge 0$, define the degree of $(x^\pm_{\alpha} \otimes a)^{(k)}$ to be $k$. For a monomial of the form $(x^\pm_{\alpha_1} \otimes a_1)^{(k_1)} \cdots (x^\pm_{\alpha_l} \otimes a_l)^{(k_l)}$ ($\alpha_1,\ldots,\alpha_l\in R^+$, $a_1,\ldots,a_l\in\mathbb{C}\langle n\rangle$, $k_1,\ldots,k_l\in\mathbb{Z}_+$, and choice of $\pm$ fixed) define its degree to be $k_1+\cdots +k_l$.

 \begin{lem}\label{commutrels} Let $\alpha,\beta\in R^+$, $i\in I$, $k,l \in \mathbb N$ and $a,b\in\mathbb{C}\langle n\rangle$.

\begin{enumerate}

\item \label{commutrels3} $\left(x^\pm_{\alpha} \otimes a\right)^{(k)}\left(x^\pm_{\beta} \otimes b\right)^{(l)}$ is in the $\mathbb Z$--span of $\left(x^\pm_{\beta} \otimes b\right)^{(l)}\left(x^\pm_{\alpha} \otimes a\right)^{(k)}$ together with monomials of degree strictly smaller than $k+l$.

\item \label{commutrels1} $$\left(x_{\alpha}^+\right)^{(l)}\left(x_{\alpha}^-\right)^{(k)} = \sum_{m=0}^{\rm{min}\{k,l\}}\left(x_{\alpha}^-\right)^{(k-m)}\binom{h_{\alpha}-k-l+2m}{m}\left(x_{\alpha}^+\right)^{(l-m)}.$$

\item \label{commutrels2} $$\binom{h_{i}}{l}\left(x^\pm_{\alpha} \otimes a\right)^{(k)} = \left(x^\pm_{\alpha} \otimes a\right)^{(k)}\binom{h_{i}\pm k\alpha(h_{i})}{l}.$$

\item \label{commutrels4} $$\left(x^\pm_{\alpha} \otimes a\right)^{(k)}\left(x^\pm_{\alpha} \otimes a\right)^{(l)} = \binom{k+l}{k}\left(x^\pm_{\alpha} \otimes a\right)^{(k+l)}.$$
\end{enumerate}
\end{lem}

\proof
(1) can be deduced from \cite[Equation (4.1.6)]{sam2}. (2) can be deduced from \cite[Lemma 26.2]{H}. (3) can be  proved by induction on $k+l$. The last item is easily established.
\endproof

\section{The main categories of modules}\label{modules}

Let $\mathbb F$ be an algebraically closed field. 

\subsection{Finite-dimensional modules for hyperalgebras \texorpdfstring{$U_\mathbb{F}(\lie g)$}{}} \label{sec:rev.g}

We now review the finite-dimensional representation theory of $U_\mathbb F(\lie g)$. If the characteristic of $\mathbb F$ is zero, then $U_\mathbb F (\lie g) \cong U(\lie g_\mathbb F)$ (cf. Remark \ref{rem1}) and the classical results can be found mainly in \cite{H}. The literature for the positive characteristic setting  is more commonly found in the context of algebraic groups, where $U_\mathbb F (\lie g)$ is known as the hyperalgebra or algebra of distributions of an algebraic group of the same Lie type as $\lie g$ (cf. \cite[Part II]{JAN}). We refer to \cite[Section 2]{JM1} for a more detailed review in the present context of hyperalgebras. 

Let $V$ be a $U_\mathbb F(\lie g)$-module. A nonzero vector $v\in V$ is called a weight vector if there exists $\mu\in U_\mathbb F(\lie h)^*$ such that $hv=\mu(h)v$ for all $h\in U_\mathbb F(\lie h)$. 
The subspace consisting of weight vectors of weight $\mu$ is called weight space of weight $\mu$ and it will be denoted by $V_\mu$. 
 When $V_\mu\ne 0$, $\mu$ it is said to be a weight of $V$ and $\wt(V) = \{\mu\in U_\mathbb F(\lie h)^*:V_\mu\ne 0\}$ is called the set of weights of $V$.
 If $V= \bigoplus_{\mu\in U_\mathbb F(\lie h)^*} V_\mu$, then $V$ is called a weight module.
 The inclusion $P\hookrightarrow U_\mathbb F(\lie h)^*$ 
gives rise to a partial order $\le$ on $U_\mathbb F(\lie h)^*$ given by $\mu\le\lambda$ if $\lambda-\mu\in Q^+$ and we have
\begin{equation}\label{e:xactonws}
(x_\alpha^\pm)^{(k)} V_\mu \subseteq V_{\mu\pm k\alpha}\quad\text{for all}\quad \alpha\in R^+, k>0,\mu\in U_\mathbb F(\lie h)^*.
\end{equation}
If $V$ is a weight-module with finite-dimensional weight spaces, its character is the function $\ch(V):U_\mathbb F(\lie h)^*\to \mathbb Z$ defined by $\ch(V)(\mu)=\dim V_\mu$. As usual, if $V$ is finite-dimensional, $\ch(V)$ can be regarded as an element of the group ring $\mathbb Z[U_\mathbb F(\lie h)^*]$ and we denote the element corresponding to $\mu\in U_\mathbb F(\lie h)^*$ by $e^\mu$. By the inclusion $P\hookrightarrow U_\mathbb F(\lie h)^*$ the group ring $\mathbb Z[P]$ can be regarded as a subring of $\mathbb Z[U_\mathbb F(\lie h)^*]$ and the action of $\cal W$ on $P$ induces an action of $\cal W$ on $\mathbb Z[P]$ by ring automorphisms with $w \cdot e^\mu = e^{w\mu}$.

If $v \in V$ is a  weight vector and $(x_\alpha^+)^{(k)} v = 0$ for all $\alpha\in R^+, k>0$, then $v$ is said to be a highest-weight vector. When $V$ is generated by a highest-weight vector, then it is said to be a highest-weight module. 

\begin{rem} The notions of lowest-weight vectors and modules are similar and obtained by replacing $(x_\alpha^+)^{(k)}$ by $(x_\alpha^-)^{(k)}$.
\end{rem} 

\begin{thm} \label{t:rh}
Let $V$ be a $U_\mathbb F(\lie g)$-module.
\begin{enumerate}[(a)]
\item If $V$ is finite-dimensional, then $V$ is a weight-module, $\wt (V) \subseteq P$, and $\dim V_\mu = \dim V_{\sigma\mu}$ for all $\sigma\in\cal W, \mu \in U_\mathbb F(\lie h)^\ast$. In particular, $\ch(V)\in\mathbb Z[P]^\cal W$.
\item \label{t:rh.weights} If $V$ is a highest-weight module of highest weight $\lambda$, then $\dim(V_{\lambda})=1$ and $V_{\mu}\ne 0$ only if $w_0\lambda \leq \mu\le \lambda$. Moreover, $V$  has a unique maximal proper submodule and, hence, also a unique irreducible quotient. In particular, $V$ is indecomposable.
\item{\label{t:rh.c}} For each $\lambda\in P^+$, the $U_\mathbb F(\lie g)$-module $W_\mathbb F(\lambda)$ given by the quotient of $U_\mathbb F(\lie g)$ by the left ideal $I_\mathbb F(\lambda)$ generated by
\begin{equation*}
U_\mathbb F (\lie n^+)^0, \quad h-\lambda(h) \quad \text{and} \quad (x_\alpha^-)^{(k)}, \quad\text{for all}\quad h\in U_\mathbb F(\lie h), \ \alpha\in R^+, \ k>\lambda(h_\alpha),
\end{equation*}
is nonzero and finite-dimensional. Moreover, every finite-dimensional highest-weight module of highest weight $\lambda$ is a quotient of $W_\mathbb F(\lambda)$.
\item If $V$ is finite-dimensional and irreducible, then there exists a unique $\lambda\in P^+$ such that $V$ is isomorphic to the irreducible quotient $V_\mathbb F(\lambda)$ of $W_\mathbb F(\lambda)$. If the characteristic of $\mathbb F$ is zero, then $W_\mathbb F(\lambda)$ is irreducible.
\item\label{t:chWg} For each $\lambda\in P^+$, $\ch(W_\mathbb F(\lambda))$ is given by the Weyl character formula. In particular, $\mu\in\wt(W_\mathbb F(\lambda))$ if, and only if, $\sigma\mu\le\lambda$ for all $\sigma\in\cal W$. Moreover, $W_\mathbb F(\lambda)$ is a lowest-weight module with lowest weight $w_0\lambda$.  \hfill \qed
\end{enumerate}
\end{thm}

The module $W_\mathbb F(\lambda)$ defined in Theorem \ref{t:rh} \eqref{t:rh.c} is called the \textit{Weyl module of highest weight $\lambda$}.

\subsection{Finite-dimensional modules for the hyper multiloop algebras \texorpdfstring{$U_\mathbb F(\lie g\langle n\rangle)$}{}} \label{sec:rev.gloop}

The results of this subsection generalize to the hyper multiloop algebras ($n>1$) some results of hyper loop algebras ($n=1$) from \cite[Section 3]{JM1}. The background for this is in \cite{CPweyl}. 

\ 

Given a $U_\mathbb F(\lie g\langle n\rangle)$-module $V$ and $\xi\in U_\mathbb F(\lie h\langle n\rangle)^*$, let
\begin{equation*}
V_\xi=\{v\in V \mid \text{ for all } x\in U_\mathbb F(\lie h\langle n\rangle), \text{ there exists } s>0 \text{ such that } (x-\xi(x))^sv = 0\}.
\end{equation*}
We say that $V$ is an $\ell$-weight module if $V = \bigoplus_{\pmb\omega\in \mathcal P_\mathbb F}^{} V_{\pmb\omega}$. In this case, regarding $V$ as a $U_\mathbb F (\lie g)$-module, we have
\begin{equation*}
V_\mu=\bigoplus_{\substack{\pmb{\omega}\in\mathcal P_\mathbb F\\ \wt(\pmb\omega)=\mu}} V_{\pmb\omega} \quad\text{for all}\quad \mu\in P \quad\text{and}\quad V=\bigoplus_{\mu\in P}^{} V_\mu.
\end{equation*}
When $V_{\pmb\omega}\ne 0$ we say that $\pmb\omega$ is an $\ell$-weight of $V$ and a nonzero element of $V_{\pmb\omega}$ is said to be an $\ell$-weight vector of $\ell$-weight $\pmb\omega$. An $\ell$-weight vector $v$ is said to be a highest-$\ell$-weight vector if $U_\mathbb F(\lie h\langle n\rangle)v=\mathbb Fv$ and $(x_{\alpha}^+\otimes b)^{(k)}v = 0$ for all for all $\alpha\in R^+$, $b\in\mathbb B\langle n\rangle$, and all $k\in\mathbb N$. When $V$ is generated by a highest-$\ell$-weight vector of $\ell$-weight $\pmb\omega$, $V$ is said to be a highest-$\ell$-weight module of highest $\ell$-weight $\pmb\omega$.

\

We begin establishing some relevant relations satisfied by certain finite-dimensional modules:

\begin{prop} \label{fdprop}
Let $V$ be a finite-dimensional $U_\mathbb F(\lie g\langle n\rangle)$-module, $\lambda\in P^+$, and $v\in V_\lambda$ be such that
$$U_\mathbb F(\lie n^+\langle n \rangle)^0v=\left(\Lambda_{i,b,r}-\omega_{i, b,r}\right)v=0,$$ $i\in I$, $ b\in\mathbb B^\bullet$, $r\in\mathbb Z_+$, and some $\omega_{i,b,r}\in\mathbb F$. Then,
\begin{itemize}
\item[i)]\label{part1} $\left(x_{\alpha,\mathbf s}^-\right)^{(k)}v=0$, for all $\alpha\in R^+$, $\mathbf s\in \mathbb Z^n$ and all $k\in\mathbb N$ with $k>\lambda(h_\alpha)$,
\item[ii)] $\Lambda_{i,b,a}v=0$, for all $a\in\mathbb N$, $a>\lambda(h_i)$,
\item[iii)] $\omega_{i,b,\lambda(h_i)}\neq0$,

\item[iv)] $\Lambda_{i,b,\lambda(h_i)}\Lambda_{i,b^{-1},r}v=\Lambda_{i,b,\lambda(h_i)-r}v$, for all $r\in\mathbb N$ with $0\leq r\leq\lambda(h_i)$.

\item[v)]  Each scalar $\omega_{i,f,r}$, with $i\in I$, $f=\mathbf t^{\mathbf a}\in\mathbb B^\bullet$, where $\mathbf a=\left(a_1,\ldots,a_n\right)\in\mathbb Z^n$ and $r\in\mathbb N$, is expressed as a  polynomial expression in $\omega_{i,t_j,\lambda(h_i)}^{-1}$ and $\omega_{i,t_j,k}$, $j\in\{1,\dots,n\}$, $k\in\mathbb N$, $0\le k \le \lambda(h_i)$.

\end{itemize}

\end{prop}

\begin{proof}
For part i), notice that for each $\mathbf s\in\mathbb Z^n$ and $\alpha\in R^+$, the elements $\left(x_{\alpha,\mp\mathbf s}^\pm\right)^{(l)}$ for all $l\in\mathbb Z_+$ generate a subalgebra $U_\mathbb F\left(\lie g\langle n\rangle_{\alpha,\mathbf s}\right)\subset U_\mathbb F(\lie g\langle n\rangle)$, which is isomorphic to $U_\mathbb F\left(\lie{sl}_2\right)$. So, the equality $\left(x_{\alpha,\mathbf s}^-\right)^{(k)}v=0$ for $\mathbf s\in \mathbb Z^n$ and $k>\lambda(h_\alpha)$ follows because $v$ generates a finite-dimensional highest-weight module for this subalgebra, which is then isomorphic to a quotient of the Weyl module $W_{\mathbb F}(\lambda(h_\alpha))$.

For parts ii) and iii), let $b=t^{\mathbf k}\in\mathbb B^\bullet$.
Therefore by Lemma \ref{basicrel} (1.1) and i) we have
$0=(-1)^a\left(x_{\alpha_i}^+\right)^{(a)}\left(x_{\alpha_i,\mathbf k}^-\right)^{(a)}v=\Lambda_{i,b,a}v$
for $a>\lambda(h_i)$ and 
$(-1)^{\lambda(h_i)}\left(x_{\alpha_i}^+\right)^{(\lambda(h_i))}\left(x_{\alpha_i,\mathbf k}^-\right)^{(\lambda(h_i))}v=\Lambda_{i,b,\lambda(h_i)}v=\omega_{i,b,\lambda(h_i)}v.$
Since $U_\mathbb F(\lie g\langle n\rangle)v$ is a finite-dimensional $U_\mathbb F(\lie g)$-module with $\mathbb Fv$ as its highest weight space, we have $\lambda-(\lambda(h_i)+m)\alpha_i$ is not a weight of $U_\mathbb F(\lie g\langle n\rangle)v$ for any $m\in\mathbb N$. Therefore, $\left(x_{\alpha_i}^-\right)^{(m)}\left(x_{\alpha_i,\mathbf k}^-\right)^{(\lambda(h_i))}v=0$ for all $m\in\mathbb N$. On the other hand, considering the subalgebra  $U_\mathbb F\left(\lie g\langle n\rangle_{\alpha_i,\mathbf k}\right)$, we see that $\left( x_{\alpha_i,\mathbf k}^-\right)^{(\lambda(h_i))}v\neq0$. Hence, $\left( x_{\alpha_i,\mathbf k}^-\right)^{(\lambda(h_i))}v$ generates a lowest weight finite-dimensional $U_{\mathbb F}(\lie g\langle n\rangle_{\alpha_i,\mathbf 0})$-module. So $\left(x_{\alpha_i}^+\right)^{(\lambda(h_i))}\left(x_{\alpha_i,\mathbf k}^-\right)^{(\lambda(h_i))}v\neq0.$ Therefore, $\omega_{i,b,\lambda(h_i)}\neq0$.

For part iv), we will expand $\Lambda_{i,b,\lambda(h_i)}\Lambda_{i,b^{-1},r}v $ using Lemmas \ref{basicrel}(1.1) and \ref{commutrels} repeatedly:

{\small $\Lambda_{i,b,\lambda(h_i)}\Lambda_{i,b^{-1},r}v $
\begin{eqnarray*}
&=&\Lambda_{i,b,\lambda(h_i)}\omega_{i,b^{-1},r}v\\
			&=&\omega_{i,b^{-1},r}\Lambda_{i,b,\lambda(h_i)}v\\
			&=&(-1)^{\lambda(h_i)}\omega_{i,b^{-1},r}\left(x_i^+\otimes b\right)^{(\lambda(h_i))}\left(x_i^-\right)^{(\lambda(h_i))}v\\
			&=&(-1)^{\lambda(h_i)}\left(x_i^+\otimes b\right)^{(\lambda(h_i))}\left(x_i^-\right)^{(\lambda(h_i))}\omega_{i,b^{-1},r}v\\
			&=&(-1)^{\lambda(h_i)}\left(x_i^+\otimes b\right)^{(\lambda(h_i))}\left(x_i^-\right)^{(\lambda(h_i))}\Lambda_{i,b^{-1},r}v\\
			&=&(-1)^{\lambda(h_i)}(-1)^r\left(x_i^+\otimes b\right)^{(\lambda(h_i))}\left(x_i^-\right)^{(\lambda(h_i))}\left(x_i^+\right)^{(r)}\left(x_i^-\otimes b^{-1}\right)^{(r)}v\\
			&=&(-1)^{\lambda(h_i)+r}\left(x_i^+\otimes b\right)^{(\lambda(h_i))}\sum_{m=0}^r\left(x_i^+\right)^{(r-m)}\binom{-h_i-\lambda(h_i)-r+2m}{m}\left(x_i^-\right)^{(\lambda(h_i)-m)}\left(x_i^-\otimes b^{-1}\right)^{(r)}v\\
			&=&(-1)^{\lambda(h_i)+r}\left(x_i^+\otimes b\right)^{(\lambda(h_i))}\sum_{m=0}^{r-1}\left(x_i^+\right)^{(r-m)}\binom{-h_i-\lambda(h_i)-r+2m}{m}\left(x_i^-\right)^{(\lambda(h_i)-m)}\left(x_i^-\otimes b^{-1}\right)^{(r)}v\\
			&&+(-1)^{\lambda(h_i)+r}\left(x_i^+\otimes b\right)^{(\lambda(h_i))}\binom{-h_i-\lambda(h_i)+r}{r}\left(x_i^-\right)^{(\lambda(h_i)-r)}\left(x_i^-\otimes b^{-1}\right)^{(r)}v\\
			&=& (-1)^{\lambda(h_i)+r}\left(x_i^+\otimes b\right)^{(\lambda(h_i))}\binom{-h_i-\lambda(h_i)+r}{r}\left(x_i^-\right)^{(\lambda(h_i)-r)}\left(x_i^-\otimes b^{-1}\right)^{(r)}v \hspace{3.7cm} (\star)\\
			&=&(-1)^{\lambda(h_i)+r}\left(x_i^+\otimes b\right)^{(\lambda(h_i))}\left(x_i^-\right)^{(\lambda(h_i)-r)}\binom{-h_i+\lambda(h_i)-r}{r}\left(x_i^-\otimes b^{-1}\right)^{(r)}v \hspace{3.5cm} (\star\star)\\
			&=&(-1)^{\lambda(h_i)+r}\left(x_i^+\otimes b\right)^{(\lambda(h_i))}\left(x_i^-\right)^{(\lambda(h_i)-r)}\left(x_i^-\otimes b^{-1}\right)^{(r)}\binom{-h_i+\lambda(h_i)+r}{r}v \\
			&=&(-1)^{\lambda(h_i)+r}\left(x_i^+\otimes b\right)^{(\lambda(h_i))}\left(x_i^-\otimes b^{-1}\right)^{(r)}\left(x_i^-\right)^{(\lambda(h_i)-r)}v\\
			&=&(-1)^{\lambda(h_i)+r}\sum_{m=0}^r\left(x_i^-\otimes b^{-1}\right)^{(r-m)}\binom{h_i-r-\lambda(h_i)+2m}{m}\left(x_i^+\otimes b\right)^{(\lambda(h_i)-m)}\left(x_i^-\right)^{(\lambda(h_i)-r)}v\\
			&=&(-1)^{\lambda(h_i)-r}\binom{h_i+r-\lambda(h_i)}{r}\left(x_i^+\otimes b\right)^{(\lambda(h_i)-r)}\left(x_i^-\right)^{(\lambda(h_i)-r)}v\\\
            &=&\binom{h_i+r-\lambda(h_i)}{r}\Lambda_{i,b,\lambda(h_i)-r}v  \\
			&=&\Lambda_{i,b,\lambda(h_i)-r}v,
\end{eqnarray*}}
where we used that $\binom{h_i+r-\lambda(h_i)}{r}$  and $\Lambda_{i,b,\lambda(h_i)-r}$ commute, $\binom{h_i+r-\lambda(h_i)}{r}v=v$, $ (-1)^{\lambda(h_i)+r}=(-1)^{\lambda(h_i)-r}$,  $(\star\star)$ is a direct application of Lemma \ref{commutrels}\eqref{commutrels2}, and $(\star)$ follows from the fact that $\lambda-(\lambda(h_i)-m+r)\alpha_i$ is not a weight of $V$ (by Theorem \ref{t:rh}(c)), which implies that $\left(x_i^-\right)^{(\lambda(h_i)-m)}\left(x_i^-\otimes b^{-1}\right)^{(r)}v=0$ for each $m\in \{0,\dots,r-1\}$.

Finally, for part v), suppose $f=\mathbf t^\mathbf a$ with $\mathbf a=(a_1,\dots,a_n)\in\mathbb Z^n$. We proceed by induction on $\alpha_\mathbf a=\sum_{i=1}^{n}|a_k|$ and $r$. If $\alpha_{\mathbf a}=1$, then $f=t_j^{\pm1}$ and we are done after possibly using item iv) in the $f=t_j^{-1}$ case. Now assume that $\alpha_\mathbf a>1$, then we express the action of $\Lambda_{i,f,r}$ in terms of the action of a sum of products of elements of the form $\Lambda_{i,g,s}$ with $g=\mathbf t^\mathbf b$ such that $\alpha_\mathbf b<\alpha_\mathbf a$ and $0\le s\le \lambda(h_i)$. This is done through the formula \eqref{aux2} below:
let $r\in\mathbb N$, $l_i=\lambda(h_i)$, $r\leq l_i$, and $g\in \mathbb B\langle n \rangle$. By Lemma \ref{basicrel}, we have
\begin{eqnarray*}
0&=&(-1)^{l_i}(x^+_{i}\otimes g)^{(l_i)}(x^-_{i}\otimes 1)^{(l_i+r)}v\\
&=&\sum_{j=0}^{l_i}\left(\left(\sum_{m=0}^\infty(x^-_i \otimes g^m)u^{m+1}\right)^{(r)} \right)_{l_i+r-j} \Lambda_{i,g,j}v\\
&=& \omega_{i,g,l_i} (x^-_i \otimes 1)^{(r)}v+
\sum_{j=0}^{l_i-1} \omega_{i,g,j} Y_jv,
\end{eqnarray*}
where $Y_j$ is the sum of monomials $\left(x^-_i \otimes g^{0}\right)^{(r_0)}\dots\left(x^-_i \otimes g^{l_i}\right)^{(r_{l_i})}$ such that $\sum_s r_s = r$, $\sum_s sr_s=l_i-j$. Now, by acting with $\left(x^+_i \otimes f\right)^{(r)}$, we obtain 
\begin{eqnarray*} 
0&=&(-1)^r\omega_{i,g,l_i}\left(x^+_i\otimes f\right)^{(r)}\left(x^-_i\otimes 1\right)^{(r)}v+(-1)^r\sum_{j=0}^{l_i-1} \omega_{i,g,j}\left(x^+_i \otimes f\right)^{(r)} Y_jv\\
   & =& \omega_{i,g,l_i}\omega_{i,f,r}v+\sum_{j=0}^{l_i-1} \omega_{i,g,j}H_jv
\end{eqnarray*}
where $H_j$ is a sum of monomials of the form $\Lambda_{i,g^{s_1}f^{c_1},b_1}\dots\Lambda_{i,g^{s_m}f^{c_m},b_m}$ with $m,b_k\in\mathbb N$ and $s_k,c_k\in\mathbb Z_+$ such that $\sum s_kb_k=l_i$, $\sum c_kb_k=r$, $s_kc_k\neq0$, and  $b_k<l_i$ for all $k\in\{1,\ldots,m\}$. In particular, taking $g=\mathbf t^{\mathbf a'}$ with $\mathbf a'=(a_1',\dots,a_n')$ defined by
$$a_i'=\begin{cases}
-a_i, &\text{ if } |a_i|< \alpha_\mathbf a ,\\
-a_i+sign(a_i)1, &\text{ if } |a_i|=\alpha_\mathbf a, \\
\end{cases}$$
we get
\begin{equation} \label{aux2}
\omega_{i,f,r}v= -(\omega_{i,g,l_i})^{-1}\sum_{j=0}^{l_i-1}\omega_{i,g,j}H_jv
\end{equation}
where $g$ and any element appearing in $H_j$ either fit the induction hypothesis or can be handled using Lemma \ref{f^kintermsoff}.
\end{proof}

Given a finite-dimensional module as in Proposition \ref{fdprop}, in light of parts ii) and iii), we see that $\omega_{i, b,r}$ can be non zero only on a finite set of $r\in\mathbb N$. Further, from part iv), 
we see that 
the action of $\Lambda_{i, b^{-1},r}$ depends on the action of $\Lambda_{i, b,k}$ $(k>0)$ for each $ b\in \mathbb B^\bullet$.  From part v), we can recover $\omega_{i,f,k}$, for all $i\in I$, $f\in\mathbb B^\bullet$, and $k\in\mathbb N$, from the action of $\omega_{i,t_j,r}$, for $1\le j\le n$ and $r\in\mathbb N$. In particular, by setting
$$\pmb\omega_{i,j}(u) = 1+\sum_{r=1}^{\lambda(h_i)} \omega_{i, t_j, r}u^r,$$ 
for all $i\in I$ and $j\in \{1,\dots,n\}$, we encode all information about the action of $U_\mathbb F(\lie h\langle n\rangle)$ (besides those related to $U_\mathbb F(\lie h)$) and we have
$$\Lambda_{i,t_j}(u)v = \pmb\omega_{i,j}(u)v.$$

\begin{rem}\label{Drinf}
Each sequence of polynomials $(\pmb\omega_{i,j})_{i\in I}^{1\le j\le n}$ for each $j\in \{1,\dots, n\}$ corresponds to a Drinfeld polynomial in the variable $t_j$ of the $\ell$-highest-weight module generated by $v$ in the sense of \cite[Section 3]{JM1}. 
\end{rem}

Inspired by this remark, the element $(\pmb\omega_{i,j})_{i\in I}^{1\le j\le n}$ will be called the \textit{$n$-Drinfeld polynomial of the highest-$\ell$-weight $U_\mathbb F(\lie g\langle n\rangle)$-module generated by $v$}.

\subsubsection{$\ell$-weight lattice} \label{lwtlattice}

We denote by $\mathcal P_\mathbb F^+$ the set consisting of all sequences $\pmb\omega = \left(\pmb\omega_{i,j}\right)_{i\in I}^{1\le j\le n}$ where each $\pmb\omega_{i,j}$ is a polynomial in $\mathbb F[u]$ with constant term $1$ and, for each $i\in I$, $\deg(\pmb\omega_{i,j})=\deg(\pmb\omega_{i,k})$ for all $1\le j,k\le n$.  Endowed with coordinatewise polynomial multiplication, $\cal P_\mathbb F^+$ is a monoid. We also denote by $\cal P_\mathbb F$ the multiplicative group associated to $\mathcal P_\mathbb F^+$ which will be referred to as the $\ell$-weight lattice associated to $\lie g$.

Let $\wt:\cal P_\mathbb F \to P^+$ be the unique group homomorphism such that $\pmb\omega = (\pmb\omega_{i,j})_{i\in I}^{1\le j \le n} \mapsto \wt(\pmb\omega) = \sum_{i\in I} \deg(\pmb\omega_{i,j})\omega_i$. 

The abelian group $\cal P_\mathbb F$ can be identified with a subgroup of the monoid of $(I\times \{1,\dots,n\})$-tuples of formal power series with coefficients in $\mathbb F$ by identifying the rational function $(1-\alpha u)^{-1}$ with the corresponding geometric formal power series $\sum_{n\geq0} (\alpha u)^n$ for all $\alpha \in \mathbb F$.  This allows us to define an inclusion $\cal P_\mathbb F \hookrightarrow U_\mathbb F(\lie h\langle n\rangle)^*$. Indeed, if $\pmb\omega=(\pmb\omega_{i,j})_{i\in I}^{1\le j\le n}\in\cal P_\mathbb F$ is such that $\pmb\omega_{i,j}(u) = \sum_{r\geq 0} \omega_{i,j,r} u^r \in \cal P_\mathbb F$,  set 
\begin{gather*}
\pmb\omega\left(\binom{h_i}{k}\right) = \binom{\wt(\pmb\omega)(h_i)}{k}, \quad \pmb\omega(\Lambda_{i,j,r})=\omega_{i,j,r}, \quad\text{for all}\quad i\in I,\  j\in \{1,\dots,n\},k,r\in\mathbb Z_+,\\ \text{and}\quad
\pmb\omega(xy)=\pmb\omega(x)\pmb\omega(y), \quad\text{for all}\quad x,y\in U_\mathbb F(\lie h\langle n\rangle).
\end{gather*}

\
 
The next theorem states some standard properties on the category of  $U_\mathbb F(\lie g\langle n\rangle)$-modules.

\begin{thm}\label{multiloop1} Let $V$ be a $U_\mathbb F(\lie g\langle n\rangle)$-module.
\begin{enumerate}
\item If $V$ is finite-dimensional, then $V$ is an $\ell$-weight module. Moreover, if $V$ is finite-dimensional and irreducible, then $V$ is a highest-$\ell$-weight module whose highest $\ell$-weight lies in $\cal P_\mathbb F^+$.
\item \label{multiloop2} If $V$ is a highest-$\ell$-weight module of highest $\ell$-weight $\pmb\omega\in\cal P_\mathbb F^+$, then $\dim V_{\pmb\omega}=1$ and $V_{\mu}\ne 0$ only if $w_0\wt(\pmb\omega)\leq \mu\le \wt(\pmb\omega)$. Moreover, $V$  has a unique maximal proper submodule and, hence, also a unique irreducible quotient. In particular, $V$ is indecomposable.
\end{enumerate}
\end{thm}

\proof In order to show (a), notice that $V$ is an $\ell$-weight module because $U_\mathbb F(\lie h\langle n \rangle)$ is commutative and this sum is just a sum of generalized eigenspaces for the action of $U_\mathbb F(\lie h\langle n\rangle)$. In case that $V$ is irreducible, the generalized eigenspaces are, in fact, eigenspaces. Further, by finite dimensionality, we must have a maximal $\ell$-weight and, then, we conclude that $V$ is a highest-$\ell$-weight module, since the $U_\mathbb F(\lie g\langle n\rangle)$-submodule generated by a highest-$\ell$-weight vector must coincide with $V$ by the irreducibility of $V$. 

For the first part of (b) it suffices to note that every vector $v \in V$ lies inside a finite-dimensional $U_\mathbb F(\lie g)$-submodule of $V$. Now all the claims
follow from the corresponding results for finite-dimensional $U_\mathbb F(\lie g)$-modules. The last part of (b) follows from standard algebraic argumentation on cyclic modules.
\endproof

\subsubsection{Evaluation modules}
We now introduce the important class of $U_\mathbb F(\lie g\langle n \rangle )$-modules
known as evaluation representations.

Given $\mathbf a=(a_1,\dots,a_n)\in (\mathbb F\setminus \{0\})^n$ and $ b=t_1^{s_1}\dots t_n^{s_n}\in \mathbb B\langle n\rangle$, we denote by $ b(\mathbf a)$ the evaluation of $b$ on $\mathbf a$, that is $ b(\mathbf a):=a_1^{s_1}\dots a_n^{s_n}$.

\begin{prop}
For $\mathbf a\in (\mathbb F\setminus \{0\})^n$, there exists a surjective algebra homomorphism ${\rm ev}_\mathbf a: U_\mathbb F(\lie g\langle n \rangle)\to U_\mathbb F(\lie g)$ mapping $(x_{\alpha}^\pm \otimes b)^{(k)}$ to $  b(\mathbf a)^{k}(x_{\alpha}^\pm)^{(k)}$ for all $\alpha \in R^+, b\in \mathbb B\langle n\rangle$. In particular, ${\rm ev}_\mathbf a(\Lambda_{\alpha,  b, r}) =(-  b(\mathbf a))^r\binom{h_\alpha}{r}$ for all $\alpha \in R^+,   b\in \mathbb B \langle n\rangle , r\in \mathbb N$.
\end{prop}

\begin{proof} From the universal property of $U(\lie g\langle n \rangle)$, there exists a unique algebra map $U(\lie g\langle n \rangle) \to U(\lie g)\otimes \mathbb C\langle n \rangle$, which is the identity on $\lie g\langle n \rangle$. Now, notice that  this map sends $U_\mathbb Z(\lie g\langle n \rangle)$ to $U_\mathbb Z(\lie g) \otimes \mathbb Z \langle n \rangle$. Hence, by reducing it modulo $p$ (i.e., tensoring with $\mathbb F$), we obtain the map ${\rm ev}: U_\mathbb F(\lie g\langle n \rangle) \to U_\mathbb F(\lie g) \otimes \mathbb F\langle n\rangle$. The other statements of the proposition are now easily deduced by direct evaluation of $(x_{\alpha}^\pm \otimes b)^{(k)}$ and $\Lambda_{\alpha,  b, r}$ on $\mathbf a$.
\end{proof}

Given any $U_\mathbb F(\lie g)$-module $V$, let $V(\mathbf a)$ be the pull-back of $V$ by ${\rm ev}_\mathbf a$.  Given $\lambda\in P^+$ and $\mathbf{a}=(a_1,\dots,a_n) \in(\mathbb F\setminus \{0\})^n$, let $\pmb\omega^{\lambda,\mathbf a }=(\pmb\omega^{\lambda,\mathbf a  }_{i,j})$ be the element of $\cal P_\mathbb F$ defined as
\begin{equation} \label{drinfeva}
\pmb\omega^{\lambda,\mathbf a  }_{i,j}(u) = (1-  a_ju)^{\lambda(h_i)} \quad\text{for all}\quad i\in I, j\in\{1,\dots,n\}.
\end{equation}

We refer to $\pmb\omega^{\omega_k,\mathbf a }=(\pmb\omega_{i,j}^{\omega_k,\mathbf a })_{i \in I}^{1\le j\le n}$, $k\in\{1,\dots,n\}$, as a fundamental $\ell$-weight. Notice that $\cal P_\mathbb F$ is the free abelian group on the set of fundamental $\ell$-weights.
Let also $\pmb\omega\mapsto \pmb\omega^-$ be the unique group automorphism of $\cal P_\mathbb F$ mapping $\pmb\omega^{\omega_i,\mathbf a }$ to $\pmb\omega^{\omega_i,\mathbf a ^{-1}}$ for all $i\in I,\mathbf a\in(\mathbb F\setminus \{0\})^n$, where $\mathbf{a}^{-1}=(a_1^{-1},\dots,a_n^{-1})$. For notational convenience we set $\pmb\omega^+=\pmb\omega$.

\begin{prop}\label{p:evmod}
Let $\lambda\in P^+$. Then, given $\mathbf a\in(\mathbb F\setminus \{0\})^n$ and a highest-weight $U_\mathbb F(\lie g)$-module $V$ of highest-weight $\lambda$, $V(\mathbf a)$ is a highest $\ell$-weight $U_\mathbb F(\lie g\langle n\rangle)$-module with Drinfeld polynomial $\pmb\omega^{\lambda,\mathbf a} \in\mathcal P_\mathbb F^+$.
\end{prop}

\proof Let $v$ be a highest-weight vector of $V$. To calculate the action of $\Lambda_{i, t_j}(u)$ under the pull-back of ${\rm ev}_\mathbf a$, we use the series which represents the logarithm, evaluate on $\mathbf a$ and compare with \eqref{drinfeva}:
$$\begin{array}{rcl}
\Lambda_{i, t_j}(u)\ v  & = & \exp \left(-\sum_{s=1}^{\infty}\frac{h_{i}\otimes t_j^{s}}{s} u^{s}\right) \ v\\
&=& 
\exp \left(-\lambda(h_i)\sum_{s=1}^{\infty}\frac{t_j^{s}}{s} u^{s}\right) \ v\\
&=&
\exp \left(\lambda(h_i)\ln(1-t_ju)\right) \ v\\
&=&(1-t_ju)^{\lambda(h_i)}  \ v\\
&=&(1-a_ju)^{\lambda(h_i)}  \ v \\
& = & (\pmb\omega^{\lambda,\mathbf a }_{i,j})(u) \ v.
\end{array}$$

\endproof

\subsubsection{Weyl modules}

We now construct an important family of objects in the category of finite-dimensional $U_\mathbb F(\lie g \langle n \rangle )$-modules and state some of their relevant properties.

\begin{thm}\label{mainthm2} Let $\pmb\omega = (\pmb\omega_{i,j})_{i\in I}^{1\le j \le n} \in \mathcal P_\mathbb F^+$ and $n\in\mathbb N$.
	\begin{enumerate}
       \item The $U_\mathbb F(\lie g \langle n \rangle )$-module, $W^n_\mathbb F(\pmb\omega)$, generated by an element $v_\omega$ with defining relations
\begin{gather*}
	U_\mathbb F(\lie n^+\langle n \rangle)^0v_{\pmb\omega}= (h- \pmb\omega(h))v_{\pmb\omega}=(x_\alpha^-)^{(k)}v_{\pmb\omega}=0.
	\end{gather*}
for all $h\in U(\lie h\langle n\rangle),\ \alpha\in R^+$, and 
$k\in\mathbb N$ with $k>\text{wt}(\pmb\omega)(h_\alpha)$ is finite--dimensional.    
    
\item Every finite-dimensional highest-$\ell$-weight-module of highest $\ell$-weight $\pmb\omega$ is a quotient of $W^n_\mathbb F(\pmb\omega)$.
  
\item If $V$ is finite-dimensional and irreducible, then there exists a unique $\pmb\omega\in \cal P_\mathbb F^+$ such that $V$ is isomorphic to the irreducible quotient $V^n_\mathbb F(\pmb\omega)$ of $W^n_\mathbb F(\pmb\omega)$. In particular, the isomorphism classes of finite-dimensional simple $U_\mathbb F(\lie g\langle n\rangle)$-modules is in bijective correspondence with $\cal P^+_\mathbb F$.

\item Let $\mu\in P$. If $W_\mathbb F^n(\pmb\omega)_\mu \ne 0$, then $W_\mathbb F^n(\pmb\omega)_{w\mu} \ne 0$ for all $w\in\mathcal W$, i.e., $$\wt(W^n_\mathbb F(\pmb\omega)) \subseteq \wt(W^n_\mathbb F(\wt({\pmb\omega}))).$$ In particular, $W_\mathbb F^n(\pmb\omega)_\mu \ne 0$ only if $w_0\wt({\pmb\omega}) \leq w\mu \leq \wt({\pmb\omega})$, for all $w \in \cal W$. 

	\end{enumerate}
\end{thm}

We call by \textit{local Weyl modules} the $U_\mathbb F(\lie g \langle n \rangle )$-module $W_\mathbb F^n(\pmb\omega)$ introduced in item (1).
This theorem establishes that the local Weyl modules are universal finite--dimensional highest-$\ell$-weight modules in the category of finite-dimensional $U_\mathbb F(\lie g \langle n \rangle)$-modules.

\begin{proof} (1) will be proven in Subsection \ref{mainthm2proof}. (2) is immediate from the definition of $W_\mathbb F^n(\pmb\omega)$ and Proposition \ref{fdprop}. (3) follows from Theorem \ref{multiloop1} and item (2).

For part (4), by the defining relations of $W_\mathbb F^n(\pmb\omega)$ we conclude that each vector $w\in W_\mathbb F^n(\pmb\omega)$ lies inside a finite-dimensional $U_\mathbb F(\lie g)$-submodule of $W_\mathbb 	F^n(\pmb\omega)$. Therefore, the statement follows from the well established corresponding result for finite dimensional $U_\mathbb F(\lie g)$-modules given in Section \ref{sec:rev.g}.

\end{proof}

\subsection{Finite-dimensional modules for hyper multicurrent algebras \texorpdfstring{$U_{\mathbb F}(\lie g[n])$}{}}\label{ssfdirreps}

The results of this subsection generalizes for hyper multicurrent algebras ($n>1$) some of the results of hyper current algebras ($n=1$) from \cite[Section 3.3 ]{bmm}, which was mainly motivated by \cite{CL}. The strategy is similar to \cite{CG}.

Let $\mathcal G_\mathbb F$ be the category whose objects are finite-dimensional $\mathbb Z_+^n$-graded $U_{\mathbb F}(\lie g[n])$-modules and the morphisms are $\mathbb Z_+^n$ graded maps of $U_{\mathbb F}(\lie g[n])$-modules. More precisely, if $V\in$ Ob $\mathcal G_\mathbb F$ then $\dim_\mathbb F V<\infty$ and 
$$V=\bigoplus_{\mathbf{r}\in\mathbb Z_+^n}V[\mathbf{r}]$$
where $V[\mathbf{r}]$ is a (finite-dimensional) subspace of $V$ such that $(x\otimes t_1^{k_1}\dots t_n^{k_n})V[\mathbf{r}]\subset V[\mathbf{r}+\mathbf{k}]$ for all $x\in\lie g$
and $\mathbf{r},\mathbf{k}\in\mathbb Z_+^n$ where $\mathbf{k}=(k_1,\ldots,k_n)$. 

For $\textbf s\in\mathbb Z_+^n$ and $V\in \mathcal G_\mathbb F$, let $\tau_\mathbf s(V)\in  \text{Ob } \mathcal G_\mathbb F$ be the $U_{\mathbb F}(\lie g[n])$-module such that $$(\tau_\mathbf sV)[\mathbf k]=V[\mathbf k+\mathbf s],$$ 
for all $\mathbf k\in\mathbb Z_+^n$. Further, for each finite dimensional $U_\mathbb F(\lie g[n])$-module, $V$, let $\textbf{ev}_0(V)\in \mathcal G_\mathbb F$ be the $U_\mathbb F(\lie g[n])$-module obtained by trivially extending the action of $U_\mathbb F(\lie g)$ to $U_\mathbb F(\lie g[n])$ (this is possible by a similar argument as in \cite[\S1 and \S3]{JM1}) by setting $U_\mathbb F(\lie g[n]_+)V=0$. Finally, for each $\mathbf r\in\mathbb Z_{+}^n$ and $\lambda\in P^+$, set $$V_\mathbb F(\lambda,\mathbf r):=(\tau_\mathbf r\circ \ev_\mathbf 0)V_\mathbb F(\lambda).$$ 

\

We can now classify the irreducible finite-dimensional $U_\mathbb F(\lie g[n])$-modules.

\begin{thm} If $V\in \mathcal G_\mathbb F$ is finite-dimensional and irreducible, then it is isomorphic to $V_\mathbb F(\lambda,\mathbf r)$ for an unique pair $(\lambda,\mathbf r)\in P^+ \times \mathbb Z_+^n$.
\end{thm}

\proof Suppose $V[\mathbf r]$ and $V[\mathbf s]$ are non-zero for some distinct $\mathbf r,\mathbf s\in \mathbb Z_+^n$. Without loss of generality, suppose that $\mathbf r>\mathbf s$ where $>$ is the natural lexicographic order in $\mathbb Z_+^n$. Then $\bigoplus_{\mathbf k \ge \mathbf r} V[\mathbf k]$ would be a proper submodule of $V$, contradicting the fact that it is irreducible. Thus there must exist a unique $\mathbf r\in \mathbb Z_+^n$ such that $V[\mathbf r]\ne 0$. Since $U_\mathbb F(\lie g[n]_+)$ changes degrees, $V = V[\mathbf r]$ must be a simple $U_\mathbb F(\lie g)$-module. From this we conclude that $V\cong  V_\mathbb F(\lambda ,\mathbf r)$ for some $\lambda \in P^+$.
\endproof

The next theorem records the basic properties of the graded analogues of local Weyl modules.

\begin{thm}\label{mainthm}
Let $\lambda \in P^+$ and $n\in\mathbb N$. Suppose characteristic of $\mathbb F$ is not $2$.
	\begin{enumerate}
     
     \item The $U_\mathbb F(\lie g[n])$-module $W_\mathbb F^n(\lambda)$ generated by the element $v_\lambda$ with defining relations
\begin{gather*}
	U_\mathbb F(\lie n^+[n])^0v_\lambda= U_\mathbb F(\lie h[n]_0)^0v_\lambda=
	(h- \lambda(h))v_\lambda=(x_\alpha^-)^{(k)}v_\lambda=0,
	\end{gather*}
for all $h\in U(\lie h),\ \alpha\in R^+,$ and 
$k\in\mathbb N$ with $k>\lambda(h_\alpha)$, is graded and finite--dimensional.

     \item If $V$ is a graded finite--dimensional $U_\mathbb F(\lie g[n])$-module generated by a weight vector $v$, of weight $\lambda$, satisfying the relations $U_\mathbb F(\lie n^+[n])^0v= U_\mathbb F(\lie h[n]_0)^0v = 0,$ then $V$ is a quotient of $W_\mathbb F^n(\lambda)$.
	\end{enumerate}
\end{thm}

We call the $U_\mathbb F(\lie g[n])$-modules, $W_\mathbb F^n(\lambda)$, introduced in item (1) \textit{graded local Weyl modules}. 

\begin{proof}

Part (1) will be proven in Section \ref{main}. To prove (2), observe that the $U_\mathbb F(\lie g)$-submodule $V'=U_\mathbb F(\lie g)v\subseteq V$ is a finite-dimensional highest-weight $U_\mathbb F(\lie g)$-module of highest weight $\lambda$. Thus, $V'$ is a quotient of $W_\mathbb F(\lambda)$ (cf. Section \ref{sec:rev.g}). The statement
follows by comparing the defining relations of $V$ and $W_\mathbb F^n(\lambda)$.
\end{proof}

\subsection{A relationship between local Weyl modules and graded local Weyl modules}\label{iso}

The result which we aim to consider here is a generalization of \cite[Theorem 1.5.2(c)]{bmm} which establishes a relationship, specifically an isomorphism, between the local Weyl modules and the graded local Weyl modules for the case of the hyper current and hyper loop algebras (that is $n=1$ in the present notation). This isomorphism can be naturally extended to the case $n>1$ as stated in the conjecture below. The proof presented in \cite[Section 5.4]{bmm} cannot be reproduced here since, for $n>1$, we do not have Demazure modules, lattice constructions, and the independence of the dimension of the local Weyl module from the base field $\mathbb F$.

It is worth to mention that this sort of relationship was first considered in the non-hyper case (i.e., characteristic zero) in \cite{fl07}.

\

Let $\lie a, \lie b$ be such that $U_\mathbb Z(\lie a)$ and $U_\mathbb Z(\lie b)$ have been defined.  Let $\mathbf a=(a_1,\dots, a_n)\in\mathbb F^n$, and $\varphi_{\mathbf a}$ the Lie algebra automorphism of $\lie g[n]_\mathbb F$ given by 
$$x\otimes f(t_1,\dots,t_n)\mapsto x\otimes f(t_1-  a_1,\dots,t_n-  a_n).$$ 
Observe that $\varphi_{\mathbf a}$ acts as the identity on $U_\mathbb F(\lie g)$ and $\varphi_{\mathbf a}$ induces an automorphism of $U_\mathbb F(\lie g[n])$, since it is just a change of variables. One easily checks that, $\varphi_{\mathbf a}(U_\mathbb F(\lie n^\pm[n]))=U_\mathbb F(\lie n^\pm[n])$ and $\varphi_{\mathbf a}(U_\mathbb F(\lie h[n]))=U_\mathbb F(\lie h[n])$. 

Denote by $res(W_\mathbb F(\pmb\omega^{\lambda,\mathbf a}))$ the module obtained by regarding $W_\mathbb F(\pmb\omega^{\lambda,\mathbf a})$ as a $U_\mathbb F(\lie g[n])$-module via restriction of the action of $U_\mathbb F(\lie g\langle n\rangle)$ to $U_\mathbb F(\lie g[n])$. Moreover, let $\varphi_{\mathbf a}^*(W_\mathbb F(\pmb\omega^{\lambda,\mathbf a}))$ be the pull-back of $res(W_\mathbb F(\pmb\omega^{\lambda,\mathbf a}))$ by $\varphi_{\mathbf a}$.

\

\begin{conj} \label{conjec}
For any $\mathbf{a} \in \mathbb F^n\setminus \{\mathbf 0\}$, $\varphi_{\mathbf a}^*(W_\mathbb F(\pmb\omega^{\lambda,\mathbf a}))$  is isomorphic to $W_\mathbb F^n(\lambda)$.
\end{conj}
\

Under the assumption that the characteristic of $\mathbb F$ is zero we present a weaker relationship in the direction of this Conjecture:

\begin{prop} \label{quotient} Suppose the characteristic of $\mathbb F$ is zero. Then, for any $\mathbf{a} \in \mathbb F^n\setminus \{\mathbf 0\}$, $\varphi_{\mathbf a}^*(W_\mathbb F(\pmb\omega^{\lambda,\mathbf a}))$  is  a quotient of $W_\mathbb F^n(\lambda)$.
 
\end{prop}

\proof  
Let $w\in W_\mathbb F(\pmb\omega^{\lambda,\mathbf a})_\lambda\setminus\{0\}$ and use the symbol $w_{\mathbf a}$ to denote $w$ when regarded as an element of $\varphi_{\mathbf a}^*(W_\mathbb F(\pmb\omega^{\lambda,\mathbf a})$. We will see in Section \ref{mainproofs}, in the proof of Theorem \ref{mainthm2}, that we have $W_\mathbb F(\pmb\omega^{\lambda,\mathbf a})=U_\mathbb F(\lie g[n])w$.  Since  $\varphi_{\mathbf a}$ is an automorphism of $U_\mathbb F (\lie g [n])$, it follows that $\varphi_{\mathbf a}^* (W_\mathbb F (\pmb\omega^{\lambda,\mathbf a}))=U_\mathbb F(\lie g[n])w_{\mathbf a}$.  Thus, we need to show that $w_{\mathbf a}$ satisfies the defining relations of $W_\mathbb F^n(\lambda)$. Since $\varphi_a$ fixes every element of $U_\mathbb F(\lie g)$, $w_{\mathbf a}$ is a vector of weight $\lambda$ annihilated by $(x_{\alpha}^-)^{(k)}$ for all $\alpha\in R^+, k>\lambda(h_\alpha)$. Further, as $\varphi_{\mathbf a}$ maps $U_\mathbb F (\lie n^+[n])$ to itself, we obtain $U_\mathbb F (\lie n^+[n])^0 w_{\mathbf a} = 0$. Therefore, it remains to show that
$U_\mathbb F(\lie h[n]_0)^0 w_{\mathbf a}=0$
which suffices to show that 
$h_i\otimes b w_{\mathbf a}=0$
for all $i\in I$ and $b\in\mathbb B[n]$, $b\ne 1$, since $U_\mathbb F(\lie h[n])$ is generated by $\{h_i\otimes b \mid \text{ for all } i\in I \text{ and } b\in\mathbb B[n]\}$ when $\rm{char} \mathbb F = 0$.
In order to do this, suppose that $b=t_1^{r_1}\dots t_n^{r_n}$. Then,  
\begin{eqnarray*}
h_i\otimes t_1^{r_1}\dots t_n^{r_n} w_{\mathbf a} &=&\varphi_{\mathbf a}^* (h_i \otimes t_1^{r_1}\dots t_n^{r_n} ) w_{\mathbf a} \\ & = & h_i \otimes (t_1-a_1)^{r_1}\dots (t_n-a_n)^{r_n} w_{\mathbf a} \\
&=& \sum_{j_1=0}^{r_1} \dots \sum_{j_n=0}^{r_n}  \binom{r_1}{j_1} \dots \binom{r_n}{j_n}  h_i \otimes (-a_1)^{j_1}\dots (-a_n)^{j_n} t_1^{r_1-j_1}\dots t_n^{r_n-j_n} w_{\mathbf a}\\
&=& \lambda(h_i)  a_1^{r_1}\dots a_n^{r_n} \sum_{j_1=0}^{r_1} \dots \sum_{j_n=0}^{r_n}  \binom{r_1}{j_1} \dots \binom{r_n}{j_n}(-1)^{j_1+\dots+ j_n}  w_{\mathbf a}
\\
&=&0
\end{eqnarray*}
where we verify that the sum in the last equation is zero by induction on $n$ and  the second to last equation follows from the fact that the irreducible quotient of $W_\mathbb F(\pmb\omega^{\lambda,\mathbf a})$ is the evaluation module with evaluation parameter $\mathbf a\in\mathbb F$ (cf. \cite[Corollary 6.1]{nss} and \cite{PB,L}), i.e,
$h_i\otimes t_1^{r_1}\dots t_n^{r_n} w=  h_i \otimes a_1^{r_1}\dots a_n^{r_n} w = a_1^{r_n}\dots a_n^{r_n} \lambda(h_i)  w$.

\endproof

\section{Main proofs}
\label{mainproofs}

\subsection{Proof of Theorem \ref{mainthm2}(1)}\label{mainthm2proof}

Set $\lambda=\wt(\pmb\omega)$. Let $v$ be a highest-$\ell$-weight vector of $W_\mathbb F^n(\pmb\omega)$. First of all, it is clear that $W_\mathbb F^n(\pmb\omega)$ is spanned by all elements of the form $\left(x_{\beta_1,\mathbf s_1}^-\right)^{(k_1)}\cdots\left(x_{\beta_m,\mathbf s_m}^-\right)^{(k_m)}v$, for $m,k_1,\ldots k_m\in \mathbb Z_+$, $\beta_1,\ldots,\beta_m\in R^+$, and $\mathbf s_1,\ldots, \mathbf s_m\in \mathbb Z^n$. In order to prove the statement, it suffices to prove that $W_\mathbb F^n(\pmb\omega)$ is spanned by the elements
$$\left(x_{\beta_1,\mathbf s_1}^-\right)^{(k_1)}\cdots\left(x_{\beta_m,\mathbf s_m}^-\right)^{(k_m)}v,$$
with $\mathbf s_1\ldots,\mathbf s_m\in \mathbb Z_+^n$ such that $\max(\mathbf s_j)<\lambda(h_{\beta_j})$ for all $j\in\{1,\ldots,m\}$ and $\sum_{j=1}^m k_j\beta_j\le \lambda-w_0\lambda$. Notice that this last condition is immediate from part (4) of Theorem \ref{mainthm2}.

Let $\cal R_\lambda = R^+\times \mathbb Z^n \times\mathbb Z_+$ and $\Xi$ be the set of functions $\xi:\mathbb N\to \cal R_\lambda$ given by $j\mapsto \xi_j=(\beta_j,\mathbf s_j,k_j)$ such that $k_j=0$ for all $j$ sufficiently large. Let $\Xi'$ be the subset of $\Xi$ consisting of the elements $\xi$ such that $0\le \min(\mathbf s_j)\le \max(\mathbf s_j)<\lambda(h_{\beta_j})$ for all $j$.  
	
Given $\xi\in\Xi$ we associate an element $v_\xi\in W_\mathbb F^n(\pmb\omega)$ as follows 
\begin{equation}\label{basicelem}
v_\xi:=\left(x_{\beta_1,\mathbf s_1}^-\right)^{(k_1)}\cdots\left(x_{\beta_m, \mathbf s_m}^-\right)^{(k_m)}v.
\end{equation}

We denote by $\mathcal S$  the $\mathbb Z$-span of vectors associated to elements in $\Xi'$.  Define the degree of $\xi$ to be $d(\xi):= \sum_j k_j$ and the maximal exponent of $\xi$ to be $e(\xi):= \max\{k_j\}$. Notice that $e(\xi)\le d(\xi)$ and $d(\xi)\ne 0$ implies $e(\xi)\ne 0$. Since there is nothing to be proved when $d(\xi)=0$ we assume from now on that  $d(\xi)>0$. Set

\begin{equation*}
\Xi_{d,e} = \{\phi\in \Xi: d(\phi)=d \text{ and } e(\phi)=e\} \qquad\text{and}\qquad \Xi_d=\bigcup\limits_{1\le e\le d} \Xi_{d,e}.
\end{equation*}
We prove by induction on $d$ and sub-induction on $e$ that, if $\xi\in\Xi_{d,e}$ is such that there exists $j$ with $\min(\mathbf s_j)<0$ or $\max(\mathbf s_{j})\ge\lambda(h_{\beta_j})$, then $v_\xi$ is in the span of vectors associated to elements in $\Xi'$. More precisely, given $0<e\le d\in\mathbb N$, we assume, by induction hypothesis, that this statement is true for every $\phi$ which belongs either to $\Xi_{d,e'}$ with $e'<e$ or to $\Xi_{d'}$ with $d'<d$. The proof is split in two cases according to whether $e=d$ or $e<d$.

First, observe that \eqref{basicrel} implies
\begin{equation}\label{basicrelv}
\left((X_{\beta;-\mathbf r,\mathbf{r+e_i}}^-(u))^{(k-\ell)}\Lambda_{\beta,\mathbf{e_i}}(u)\right)_kv = 0 \qquad \forall\ \mathbf r \in\mathbb Z^n,\beta\in R^+, k,\ell\in\mathbb Z, k>\lambda(h_\beta), 1\le \ell\le k,
\end{equation}
where $\mathbf{e_i}$ is $i^{th}$-canonical element in $\mathbb Z^n$.

\begin{itemize}
\item \textbf{Case $e=d$.} In this case, it follows that $v_\xi= (x_{\beta,\mathbf s}^-)^{(e)}v$ for some $\beta\in R^+$ and $\mathbf s\in\mathbb Z^n$.  

Let $l=e\lambda(h_{\beta})$ and $k=l+e$ in \eqref{basicrelv} to obtain
\begin{equation}\label{basicrelv2}
\sum_{n=0}^{\lambda(h_{\beta})} (x_{\beta,\mathbf r+(n+1)\mathbf{e_i}}^-)^{(e)}\Lambda_{\beta,\mathbf{e_i},l-en}v +  \substack{\text{ other terms in the span of vectors } \\ v_{\phi'} \text{ with } \phi'\in \Xi_{e,e'} \text{ for } e'<e} = 0.
\end{equation}
Suppose that $\mathbf s=(s_1,\dots,s_n)$. For each $i\in\{1,\dots, n\}$ we consider the cases $s_i\ge \lambda(h_{\beta})$ and $s_i<0$ separately and prove the statement by a further induction on $s_i$ and $|s_i|$, respectively, with a suitable choice of $\mathbf r$ in the equation above: 
\begin{itemize}
    \item[i)] if $s_i\ge \lambda(h_{\beta})$, we deal with \eqref{basicrelv2}  with $\mathbf r= \mathbf s-(\lambda(h_{\beta})+1)\mathbf{e_i}$;
    
    \item[ii)] if $s_i<0$, observe that  $n_0=\min\{j\in\mathbb Z_+\mid j\le \lambda(h_\beta) \text{ and } \Lambda_{\beta,\ell-ej}v\ne 0\}$ is well defined once we have  $\Lambda_{\beta,\ell-e\lambda(h_{\beta})}v=v\ne 0$, then we deal with \eqref{basicrelv2} starting from $n=n_0$ and  
    $\mathbf r= \mathbf s-(n_0+1)\mathbf{e_i}$.
\end{itemize}

\ 

\item Case $e<d$. We can assume, by the induction hypothesis, that $0<\min(\mathbf{s_j})\le \max(\mathbf{s_j})<\lambda(h_{\beta_j})$ for $j>1$ in \eqref{basicelem}. Now, an application of Lemma \ref{commutrels}\eqref{commutrels3} completes the inductive argument. \hfill \qedsymbol
\end{itemize}

\subsection{Proof of Theorem \ref{mainthm}(1).} \label{main}

We begin with a technical property of graded local Weyl $U_\mathbb F(\lie{sl}_2[n])$-modules and a technical property of certain finite--dimensional $U_\mathbb F(\lie g[n])$-modules. These properties are established for graded local Weyl $U_\mathbb F(\lie{sl}_2[n])$-modules and then extended to suitably chosen subalgebras of $U_\mathbb F(\lie g[n])$.

\begin{lem} \label{auxlem}
Suppose  the characteristic of $\mathbb F$ is not 2.\
\begin{enumerate}[(a)]

\item \label{a} If $m \in \mathbb Z_+$, $v$ is a nonzero vector of weight $m\omega$ in the local Weyl $U_\mathbb F (\lie{sl}_2 [n])$-module $W_\mathbb F(m)$, and $\mathbf a\in\mathbb{Z}_+^n$ with $\max\textbf a\geq m$, then $\left(x^-\otimes t^{\mathbf a}\right)^{(k)}v = 0$, for all  $k \in \mathbb N$. 
       
\item \label{b}  If $V$ is a finite-dimensional $U_\mathbb F(\lie g[n])$-module, $\lambda \in P^+$, and $v \in V_\lambda$ is such that
		\[
		U_\mathbb F (\lie n^+[n])^0 v = U_\mathbb F (\lie h[n]_0)^0 v = 0,
		\]
then $(x_\alpha^-\otimes t^{\mathbf b})^{(k)}v = 0$ for all $\alpha \in R^+$, $k \in \mathbb N$, and $\mathbf b\in\mathbb{Z}_+$ with $\max\mathbf b\geq\lambda(h_\alpha)$.

\end{enumerate}
\end{lem}

\proof
\begin{enumerate}[(a)]
\item We will first consider the case where one of the $a_i\geq m$ and all of the others are $0$. Without loss of generality we are assuming now that $a_2=\dots=a_n=0$ and $a_1\geq m$. To prove this case we will proceed by induction on $k$. When $k=1$, using Lemma \ref{basicrel}, we obtain
\begin{align*}
    0&=(-1)^{a_1}\left(x^+\otimes t_1\right)^{(a_1)} (x^-\otimes1)^{(a_1+1)}v\\
    &=\left(X_{t_1,1}(u)\right)_{a_1+1}v\\
    &=\left(\sum_{j\ge0}\left(x^- \otimes t_1^j\right)u^{j+1}\right)_{a_1+1}v\\
    &=\left(x^-\otimes t_1^{a_1}\right)v.
\end{align*}

If $k>1$, by using Lemma \ref{basicrel} again, we get
\begin{align}\label{eq2}
0&=(-1)^{ka_1}\left(x^+\otimes t_1\right)^{(ka_1)}(x^- \otimes1)^{(ka_1+k)}v \nonumber\\
&=\left(\left(X_{t_1,1}(u)^{(k)}\right)\right)_{ka_1+k}v\\
&=\left( \left(\sum_{j\ge0}\left(x^- \otimes t_1^{j}\right)u^{j+1}\right)^{(k)}\right)_{ka_1+k}v. \nonumber
\end{align}
Given $l\in\mathbb N$ and ${\bf k}=\left(k_0,k_1,\dots,k_l\right)\in\mathbb Z_+^{l+1}$ set 
$X_{\bf k}=\prod_{j=0}^l\left(x^-\otimes t_1^j\right)^{(k_j)}.$
From \eqref{eq2}, we get
$\left(x^-\otimes t_1^{a_1}\right)^{(k)}v=-\sum_{\bf{k}}X_{\bf{k}}v,$
where the sum on the right side is over ${\bf k}\in\mathbb Z^{l+1}_+$ such that $\sum_{j=0}^lk_j=k$, $\sum_{j=0}^ljk_j=a_1k$, and $k_j>0$ for more than one $j\in\{0,\ldots,l\}$. In particular, we have $l\geq a_1$ and $k_j<k$ for all $j\in\{0,\dots,l\}$. Hence, by the induction hypothesis, $\sum_{\bf k}X_{\mathbf k}v=0$. Therefore, $\left(x^-\otimes t_1^{a_1}\right)^{(k)}v=0$ for all $k\in\mathbb N$ if $a_1\geq m$.
 
Now for the case of $a_1,\ldots,a_n\in\mathbb Z_+$ with $a_i\geq m$ for some $i\in\{1,\ldots,n\}$ and at least one $a_j\neq 0$ for $i\neq j$. Without loss of generality we may assume that $a_1\geq m$. Then,
\begin{eqnarray*}
\left(x^-\otimes t_1^{a_1}\dots t_n^{a_n}\right)^{(k)}v
&=&\left(-\frac{1}{2}\left[h\otimes t_2^{a_2}\dots t_n^{a_n},x^-\otimes t_1^{a_1}\right]\right)^{(k)}v
\\
&=&\left(-\frac{1}{2}\right)^k\left(\left(h\otimes t_2^{a_2}\dots t_n^{a_n}\right)\left(x^-\otimes t_1^{a_1}\right)-\left(x^-\otimes t_1^{a_1}\right)\left(h\otimes t_2^{a_2}\dots t_n^{a_n}\right)\right)^{(k)}v
\\
&=&0,
\end{eqnarray*}
because $\left(x^-\otimes t_1^{a_1}\right)v=0$, by the previous case, and $\left(h\otimes t_2^{a_2}\dots t_n^{a_n}\right)v=0$, by the definition of the local Weyl module.

\item The set $\left\{\left(x_\alpha^\pm \otimes f\right)^{(k)}\ \big|\  k\in \mathbb Z_+,\ f\in\mathbb{C}[n]\right\}$ generates a subalgebra $U_\mathbb F\left(\lie{sl}_\alpha[n]\right) \subseteq U_\mathbb F(\lie g[n])$ isomorphic to $U_\mathbb F(\lie {sl}_2[n])$. The vanishing of $\left(x^-_\alpha \otimes t_1^{b_1}\dots t_n^{b_n}\right)^{(k)}v$, for all $b_1,\ldots,b_n\in\mathbb Z_+$ with $b_i\geq\lambda(h_\alpha)$ for some $i\in\{1,\ldots,n\}$ and $k \in \mathbb N$, follows from $(a)$ and the fact that $v$ generates a finite-dimensional highest-weight module for $U_\mathbb F\left(\lie{sl}_\alpha[n]\right)$, which is isomorphic to a quotient of the graded local Weyl $U_\mathbb F(\lie {sl}_2[n])$-module $W_\mathbb F(\lambda(h_\alpha))$.
\end{enumerate}
\endproof

We now use Lemma \ref{auxlem} to prove Theorem \ref{mainthm}(1):

\

\noindent
We proceed by using the notation and construction used in the proof of Theorem \ref{mainthm2}(1) \textit{mutatis mutandis}. Let $v$ be a highest-weight vector of $W_\mathbb F^n(\lambda)$. In this context, it suffices to prove that $W_\mathbb F^n(\lambda)$ is spanned by the elements
$$\left(x_{\beta_1,\mathbf s_1}^-\right)^{(k_1)}\cdots\left(x_{\beta_m,\mathbf s_m}^-\right)^{(k_m)}v,$$
with $\mathbf s_1\ldots,\mathbf s_m\in \mathbb Z_+^n$ such that $\max(\mathbf s_j)<\lambda(h_{\beta_j})$ for all $j\in\{1,\ldots,m\}$ and $\sum_{j=1}^m k_j\beta_j\le \lambda-w_0\lambda$. Notice that this last condition is immediate from Theorem \ref{mainthm}(2). 

Let $\cal R_\lambda = R^+\times \mathbb Z_+^n \times\mathbb Z_+$ and consider $\Xi$, $\Xi'$, $v_\xi$, $\mathcal S$, $d(\xi)$, $e(\xi)$, and $\Xi_{k}$ as in the proof of Theorem \ref{mainthm2}(1). Given $\xi \in \Xi$, once again we can assume that $d(\xi)>0$ and we prove, by induction on $d\in\mathbb{N}$, that if $\xi\in\Xi_d$ and there exists $j\in\mathbb{N}$ with $\max(\mathbf s_j) \ge \lambda(h_{\beta_{j}})$, then $v_\xi\in \mathcal S$.
	
Let $\xi \in \Xi_d$. If $d=1$, it follows that $v_\xi=\left(x_{\alpha,\mathbf s}^-\right) v$ for some $\alpha\in R^+$ and $\mathbf s \in\mathbb Z_+^n$. In this case, from Lemma \ref{auxlem}\eqref{b}, we conclude that $v_\xi=0 \in\mathcal S$ whenever $\max(\mathbf s) \ge \lambda(h_\alpha)$ or $\alpha >\lambda -w_0\lambda$. On the other hand, if $\max(\mathbf s) < \lambda(h_\alpha)$ and $\alpha\le\lambda - w_0\lambda$, then $v_\xi=\left(x_{\alpha,\mathbf s}^-\right)v\in \mathcal S$ by the definition of $\mathcal S$.
	
Now, let $d\in\mathbb N$ and suppose the result is true for all $\xi \in\Xi_{d'}$ with $d' < d$. We split the argument in two case: $e(\xi)=d(\xi)=d$ or $e(\xi)<d(\xi)=d$. When $e(\xi)=d(\xi)=d$, it follows that $v_\xi=\left(x_{\alpha,\mathbf{k}}^-\right)^{(d)}v$ for some $\alpha\in R^+$ and $\mathbf{k}\in\mathbb Z_+^n$. If $\max{\mathbf k}< \lambda(h_\alpha)$, we have $v_\xi\in \mathcal S$ by the definition of $\mathcal S$. If $\max{\mathbf k}\ge \lambda(h_\alpha)$, then $v_\xi\in \mathcal S$ because $v_\xi=0$ by Lemma \ref{auxlem}\eqref{b} again. Finally, for the case $e(\xi)<d(\xi)=d$ we can assume, by the induction hypothesis, that $\max(\mathbf{s_j})<\lambda(h_{\beta_j})$ for $j>1$. Now, applying Lemmas \ref{commutrels}\eqref{commutrels3} and \ref{auxlem}\eqref{b} we complete the inductive argument.\hfill \qedsymbol

\vfill

\end{document}